\documentclass[11pt,reqno]{amsart}

\usepackage{graphicx,enumerate, stmaryrd, color, url}
\usepackage{amsmath}
\usepackage{amssymb}
\usepackage{amsthm}
\usepackage{amsfonts}
\usepackage{epsfig,amscd,amsthm,mathtools,fourier}
\usepackage{tikz-cd}
\usepackage{tikz}
\usepackage[all]{xy}
\usepackage{comment}
\usepackage{times}
\usepackage{bm}
\usepackage[hidelinks]{hyperref}
\usepackage{appendix}
\usepackage{float}
\usepackage[margin=1.3in,marginparwidth=1.1in,marginparsep=0.1in]{geometry}
\usepackage{marginnote}
\usepackage{import}
\usepackage[normalem]{ulem}
\usepackage{diagbox}

\newtheorem{thm}{Theorem}[section]
\newtheorem{fact}[thm]{Fact}
\newtheorem{lma}[thm]{Lemma}
\newtheorem{prop}[thm]{Proposition}

\newtheorem{cor}[thm]{Corollary}
\newtheorem*{ThmA}{Theorem A}
\newtheorem*{ThmB}{Theorem B}
\newtheorem*{ThmC}{Theorem C}
\newtheorem*{ThmD}{Theorem D}

\theoremstyle{definition}
\newtheorem{defn}[thm]{Definition}

\newtheorem{rmk}[thm]{Remark}

\newtheorem{ex}[thm]{Example}

\usepackage{xcolor,color,soul}

\newcommand{\W}{W_{\mathrm{Dhar}}}

\newcommand{\un}{\underline}

\newcommand{\md}{\underline{d}}

\newcommand{\cX}{\mathcal{X}}

\newcommand{\Spec}{\operatorname{Spec}}

\newcommand{\val}{\operatorname{val}}

\newcommand{\cO}{\mathcal{O}}

\newcommand{\Pic}{{\operatorname{{Pic}}}}
\newcommand{\Prin}{{\operatorname{{Prin}}}}

\newcommand{\Div}{{\operatorname{{Div}}}}

\newcommand{\Br}{\mathrm{Br}}

\newcommand{\PP}{\mathbb{P}}
\newcommand{\ZZ}{\mathbb{Z}}

\newcommand{\ud}{\underline{d}}

\newcommand{\rmax}{r^{\max}}

\newcommand{\rmaxx}{r^{\MAX}}
\newcommand{\ralg}{r^{\alg}}
\newcommand{\ralgg}{r^{\ALG}}

\DeclareMathOperator{\MAX}{MAX}
\DeclareMathOperator{\ALG}{ALG}
\DeclareMathOperator{\alg}{alg}

\title{Clifford representatives via the uniform algebraic rank} 

\author[M. Barbosa, K. Christ, M. Melo]{Myrla Barbosa, Karl Christ and Margarida Melo}

\thanks{MB was supported by CNPq, grant 170255/2023-9. KC was partially supported by NSF FRG grant DMS–2053261 and is member of the INdAM group GNSAGA.  MM is supported by MIUR via the projects  PRIN2017SSNZAW (Advances in Moduli Theory and Birational Classification),  PRIN 2022L34E7W (Moduli spaces and birational geometry) and PRIN 2020KKWT53 (Curves, Ricci flat Varieties
and their Interactions), and is a member of the Centre for Mathematics of the University of
Coimbra -- UIDB/00324/2020, funded by the Portuguese Government through FCT/MCTES and of the INdAM group GNSAGA }

\address[Barbosa]{Departamento de Matematica MTM, Universidade Federal de Santa Catarina, 88.040-900, Florianopolis-SC, Brazil}\email{myrlakedynna@gmail.com}
\address[Christ]{Dipartimento di Matematica\\
	Università di Torino\\Via Carlo Alberto 10 \\10123 Turin\\  Italy }\email{karl.christ@unito.it}
\address[Melo]{Department of Mathematics and Physics, Universit\`a Roma Tre, Largo San Leonardo Murialdo, 00146 Roma, Italy}\email{margarida.melo@uniroma3.it}

\begin{document}

\begin{abstract}
    In this paper, we introduce the uniform algebraic rank of a divisor class on a finite graph. We show that it lies between Caporaso's algebraic rank and the combinatorial rank of Baker and Norine. We prove the Riemann-Roch theorem for the uniform algebraic rank, and show that both the algebraic and the uniform algebraic rank are realized on effective divisors. As an application, we use the uniform algebraic rank to show that Clifford representatives always exist. We conclude with an explicit description of such Clifford representatives for a large class of graphs.
\end{abstract}

\maketitle

\section{Introduction}
Describing the limits of sections of line bundles $\mathcal L_\eta$ on a smooth curve $\mathcal X_\eta$ when the curve degenerates to a nodal curve $X$ is a notoriously difficult problem in algebraic geometry. Among the many approaches to address this question, divisor theory on graphs has  allowed to give combinatorial proofs of important algebro-geometric results such as the Brill-Noether theorem, the Petri theorem, the maximal rank conjecture and the birational geometry of the moduli space of curves; see \cite{JPsurvey} for a survey of these results.

Divisor theory on graphs has been developed in the last 20 years following the breakthrough work of Baker and Norine \cite{BN}, who introduced the notion of rank $r_G(\delta)$ of a divisor class $\delta$ on a graph $G$. Then they showed that, quite remarkably, this rank satisfies several classical theorems from algebraic geometry, such as the Riemann-Roch theorem and the Clifford inequality.

The divisors $\ud$ in a divisor class $\delta$ on a graph $G$ are formal linear combinations of vertices of $G$ and can be interpreted as combinatorial types of line bundles on nodal curves with dual graph equal to $G$ (their multidegrees). Any line bundle $\mathcal L_\eta$ on the general fiber $\mathcal X_\eta$ of a regular smoothing of a nodal curve $X$ extends to a line bundle $L$ on $X$ with multidegree in some divisor class $\delta$ on the dual graph $G$ of $X$. What makes the Baker-Norine rank useful in algebraic geometry is Baker's specialization lemma \cite{B}: \begin{equation} \label{eq:specialization} r(\mathcal X_\eta, \mathcal L_\eta) \leq r_G(\delta).\end{equation} This inequality can be strict, and much effort has been devoted to describing the gap; see, for example, \cite{C1}, \cite{AB}, \cite{FJP}, and \cite{AG22} for refinements of the rank and Figure~\ref{fig1} in Section~\ref{sec:other ranks} below for a discussion of their relation.

In this paper, our main point of reference is Caporaso's algebraic rank $\ralg(G, \delta)$ introduced in \cite{C1}. It is defined via a min-max construction, as follows:
\[\ralg(G,\delta):=\max_{X\in M_G } \left\{ \min_{\ud \in \delta} \left \{ \max_{L\in \Pic^{\ud}(X)} \left \{r(X,L) \right\} \right\}\right \},\]
where $M_G$ denotes the set of isomorphism classes of curves with dual graph $G$, and $\Pic^{\ud}(X)$ the set of line bundles on $X$ of combinatorial type $\ud$. 

Among the attempts to describe the discrepancy in \eqref{eq:specialization}, the algebraic rank is unique in so far as it is defined purely in terms of line bundles on the limit curve $X$. Nonetheless, in \cite{CLM}, Caporaso, Len and the third author were able to show that the algebraic rank is bounded from above by the Baker-Norine rank. On the other hand, it follows from upper-semicontinuity of the algebro-geometric rank that it is bounded from below by $r(\mathcal X_\eta, \mathcal L_\eta)$. Thus the algebraic rank $\ralg(G, \delta)$ refines Baker's inequality \eqref{eq:specialization}. Further properties of the algebraic rank have been established in \cite{KY15, KY16, Len17} and we give a summary in Fact~\ref{fact:eqrgralg}. 

In the current paper, we propose to modify the definition of the algebraic rank and to study the \emph{uniform algebraic rank}, defined by

\begin{equation*}\ralgg(G,\delta):=\min_{\ud \in \delta} \left \{\max_{X\in M_G} \left \{ \max_{ L \in \Pic^{\ud}(X)} \left \{ r(X,L) \right \} \right \} \right \}. \end{equation*}

Our first main result is that this notion of rank refines Baker's specialization \eqref{eq:specialization} further.
\begin{ThmA}
    Let $X$ be a nodal curve with dual graph $G$, $\mathcal X_\eta$ a regular one-parameter smoothing of $X$ and $L_\eta$ a line bundle on $\mathcal X_\eta$ that specializes to a divisor class $\delta$ on $G$. Then:    \begin{equation}\label{eq:inequalities} r(\mathcal X_\eta, \mathcal{L}_\eta) \leq \ralg(G, \delta) \leq \ralgg(G, \delta) \leq r_G(\delta).\end{equation}
\end{ThmA}

See  Proposition~\ref{prop:ineqralgalgg} and Theorem~\ref{thm:mainineq}. By \cite{Len17}, the first inequality can be strict, and it is not difficult to see that also the third inequality can be strict, see Remark~\ref{rmk:strict ineq}. Constructing an example where the second inequality is strict is more difficult, and we do so in Example~\ref{ex:strict ineq}.

In a forthcoming paper of the first and third authors, we will describe a further refinement of the uniform algebraic rank, which takes into account certain line bundles on quasistable modifications of curves in $M_G$ (i.e., nodal curves obtained by inserting exceptional rational curves on the preimages of partial normalizations of the original curve). This modified rank is inspired by the geometry of compactified Jacobians and appeared first in the PhD thesis of the first author \cite{B21}, where it is shown to be equal to the Baker-Norine rank in some cases where the third inequality in \eqref{eq:inequalities} is strict.

Next, we establish the Riemann-Roch theorem for the uniform algebraic rank in Theorem~\ref{rrminmax}:
\begin{ThmB}[Riemann-Roch for the uniform algebraic rank]
    Let $\ud$ be a divisor of degree $d$ on a graph $G$ of genus $g$. Denote by $\un k_G$ the canonical divisor on $G$. Then
    \[\ralgg(G,[\ud]) - \ralgg(G, [\un{k}_G - \ud] ) = d - g + 1 ,\]
where $[\ud]$ represents the class of the divisor $\ud$.
\end{ThmB}

\medskip

Our main application of these constructions is the following: 
A divisor $\md \in \delta$ is called a \emph{Clifford representative} if every line bundle $L$ of combinatorial type $\md$ on any curve $X$ with dual graph $G$ satisfies the Clifford inequality, $r(X,L) \leq \frac{d}{2}$. The following result answers \cite[Question 4.7]{CLM} affirmatively: 

\begin{ThmC}
    Let $\delta$ be a divisor class of degree $d$ on a graph $G$. If  $0 \leq d \leq 2g - 2$, then $\delta$ contains a Clifford representative.
\end{ThmC}

The existence of Clifford representatives in any divisor class is far from obvious, as it is well-known that even for reasonably well-behaved divisors $\md$, the Clifford inequality can be violated by some line bundles, cf. \cite[\S 4.3]{caporasosemistable} and \cite{Clifford}. Following the idea of \cite{CLM} for the algebraic rank, our proof combines the inequality in \eqref{eq:inequalities} with the fact due to Baker and Norine \cite{BN} that their combinatorial rank $r_G(\delta)$ satisfies the Clifford inequality. In particular, the argument is non-constructive and the remainder of the paper is devoted to the construction of explicit Clifford representatives. 

As is clear from the argument above, any divisor $\md$ that realizes the uniform algebraic rank is a Clifford representative.
Let us compare the question of constructing such representatives to another class of representatives $\md \in \delta$, the semibalanced divisors used in the construction of universal compactified Jacobians (see Definition~\ref{balanced}). Any divisor class $\delta$ admits semibalanced representatives, and they are explicit divisors that minimize the minimal rank of line bundles on curves with dual graph $G$ in $\delta$ by \cite[Theorem 1.2]{Chr23}. Our motivation here, on the other hand, is to find representatives that minimize the maximal rank of line bundles on curves with dual graph $G$ in a fixed divisor class $\delta$. Since semibalanced divisors are in general not Clifford representatives \cite[\S 4.3]{caporasosemistable}, these notions do not coincide.

Our main result regarding the question of finding divisors that realize the uniform algebraic rank is Theorem~\ref{thm:effective realization}. It shows that to calculate the algebraic and uniform algebraic rank one can restrict to effective divisors $\md$, that is, divisors with non-negative value on each vertex of $G$ (notice that there are only finitely many of these in a given divisor class):
\begin{ThmD}
    Let $\delta$ be an effective divisor class on a graph $G$.
    Then both the algebraic rank $\ralg(G, \delta)$ and the uniform algebraic rank $\ralgg(G, \delta)$ can be realized by an \emph{effective representative} $\md \in \delta$. 
\end{ThmD}
Notice that if $\delta$ is not effective, then it follows from \eqref{eq:inequalities} that $\ralg(G,\delta)=\ralgg(G,\delta)=-1$.

Concretely, Theorem D states that there is an effective divisor $\md \in \delta$ and a curve $X$ with dual graph $G$ such that \[\ralg(G, \delta) = \max_{L \in \Pic^{\md}(X)}\left \{r(X, L) \right\}\] and that there is an effective divisor $\md' \in \delta$ such that 
\[\ralgg(G, \delta) = \max_{X \in M_G} \left \{\max_{L \in \Pic^{\md'}(X)}\{r(X, L)\} \right \}.\]
In particular, for the uniform algebraic rank, this jointly with equation \eqref{eq:inequalities}, implies that for all algebraic curves $X$ with dual graph $G$ and for all line bundles $L$ on $X$ of combinatorial type $\md'$,  $r(X, L)\leq \ralgg(G, \delta)\leq r_G(\delta)$.

Finally, we construct explicit Clifford representatives for a large class of graphs in Theorem~\ref{thm:main_detailed} using previous results from \cite{Clifford, Chr23} (though they need not in general realize the uniform algebraic rank). 
In particular, this includes all graphs without bridges and whose vertex weights are different from $0$. Previously, such a construction was only known for $d = 0$ or $d = 2g -2$, if $G$ has no bridges and $d \leq 4$, or $G$ has at most $2$ vertices by \cite{caporasosemistable}.

\medskip

\noindent 
{\bf Acknowledgements.} It is our pleasure to thank Lucia Caporaso for her helpful comments on a previous version of the paper. 

\section{Notation and preliminaries}

\subsection{Graphs and divisors}
\label{subsec:graphs and divisors}

Throughout the paper we will denote with $G=(V, E, \omega)$ a finite vertex-weighted graph, where $V=V(G)$ denotes the set of vertices of $G$, $E=E(G)$ its set of edges and $\omega:V \to \ZZ_{\geq 0}$ its weight function. If $\omega=0$, $G$ is called \emph{weightless}. The graph $G$ may contain multiple edges or loops. Unless otherwise stated, we will assume that $G$ is connected.

We will denote with $\val (v)$ the valence of a vertex $v$, i.e., the number of edges adjacent to $v$, with loops counting twice.

The \emph{genus} of a connected graph $G$ is
$$g = g(G):= |E(G)|-|V(G)|+1 + \sum_{v\in V(G)}\omega(v).$$

\begin{defn} \label{stablegraph} Let $G$ be a connected graph of genus $g \geq 2$. We say that $G$ is a \emph{stable} (respectively \emph{semistable}) graph if every vertex of weight zero has valence at least $3$  (respectively $2$). 
\end{defn}

Fixing an ordering $V(G)=\{v_1,\ldots,v_\lambda\}$,
we denote by $\Div(G)$ the free $\ZZ$-module generated by elements of $V(G)$, i.e.,
$$\Div(G):=\Bigg\{\ud=\sum_{i=1}^{\lambda} d_iv_i, \, d_i\in \ZZ\Bigg\}\cong\ZZ^{\lambda}.$$

The \textit{degree} of $\ud=(d_1,\ldots, d_{\lambda})$ is the integer $| \ud |:=\sum_{i=1}^{\lambda} d_i$. 
A divisor $\ud$ is \textit{effective} if $d_i \geq 0$ for all $v_i \in V$. In this case we write $\ud\geq 0$. The subset of divisors of degree $d$ is denoted by $\Div^{d}(G)$ and $\Div_+(G)$ denotes the subset of effective divisors in $\Div(G)$. Furthermore, the subset of effective divisors of degree $d$ is denoted by $\Div^d_+(G)$. Given $Z\subseteq V(G)$ we write $\ud(Z)= \sum_{v_i \in Z} d_i.$

The \textit{canonical divisor} $\underline{k}_G$ on $G$ is the divisor with value at a vertex $v$ of $G$ given by
$$\underline{k}_G(v)=2\omega(v) - 2 + \val(v).$$
Notice that $|\underline{k}_G| = 2 g(G) - 2$.

The group $\Div(G)$ is endowed with an intersection product associating an integer, written $\ud_1\cdot \ud_2$, to $\ud_1,\ud_2\in \Div(G)$.
It is given by linearly extending the following rule on vertices: 
If $v_1\neq v_2$ we set  $v_1 \cdot v_2$ to be equal to the number of edges joining $v_1$ with $v_2$, whereas \linebreak
$v_1\cdot v_1=-\sum _{v\in V\smallsetminus \{v_1\}}v\cdot v_1$.

Given a subset $Z\subset V(G)$, we set $Z^{c}:= V(G)\setminus Z$ and define the divisor $\underline{t}_Z$ such that
$$\underline{t}_Z(v) := \begin{cases}
v\cdot Z & \text{ if } v\notin Z \\
-v\cdot Z^c & \text{ if } v\in Z,
\end{cases}$$
where we identify $Z\subset V$ with the divisor $\sum_{v\in Z}v$.
Divisors of the form $\underline{t}_Z$ generate a subgroup of $\Div^0(G)$, denoted by $\Prin(G)$, and whose elements are called \textit{principal divisors}. We say that two divisors $\ud$ and $\ud'$ are \textit{linearly equivalent} if their difference is a principal divisor, and we write $\ud \sim\ud'$.  We define the Picard group\footnote{The Picard group of a graph is known in the literature under many different names, as the degree class group, the Laplacian group, the sandpiles group, etc. It is a consequence of Kirchhoff's matrix tree theorem that, for fixed degree $d$, the cardinality of $\Pic^d(G)$ is finite and equals the number of spanning trees of the graph $G$ (see \cite{BMS} for some properties of $\Pic(G)$ and a proof of the matrix tree theorem.)} of $G$
$$\Pic (G) = \Div (G) /\sim.$$ The equivalence class of a divisor $\ud$ is denoted by $[\ud]$. We also use the notation $\delta$ for an element of $\Pic(G)$, and we write $\ud \in \delta$ if $\ud$ is a representative. Since the principal divisors have degree zero, equivalent divisors have the same degree. We set, for an integer $d$,
$$\Pic^d(G)=\Div^d(G) / \sim.$$
We call a divisor class $\delta$ effective if it contains an effective representative $\md$.

\subsection{Semibalanced and reduced divisors}
\label{sec:semibalanced}

In what follows, we will discuss two important classes of representatives in a given equivalence class $\delta \in \Pic(G)$. The first are the semibalanced divisors that are the multidegrees of semistable line bundles in the definition of universal compactified Jacobians; that is, divisors that appear as multidegrees of line bundles that are semistable with respect to the canonical polarization. Their definition is purely combinatorial, and we refer, for example, to \cite{Caporasocompactification, MF12} for their connection to compactified Jacobians.

\begin{defn} \label{balanced} Let $G$ be a semistable graph of genus $g\geq 2$, and let $\ud \in \Div^d G$. Given a subset $Z \subset V(G)$, we define the parameters
$$m_Z(d) := d \:  \frac{ \underline{k}_G(Z)}{2g-2} - \frac{(Z\cdot Z^c)}{2} \mbox{ and } M_Z(d):=  d \: \frac{ \underline{k}_G(Z)}{2g-2} + \frac{(Z\cdot Z^c)}{2}.$$
 We say that $\ud$ is \textit{semibalanced} if for every $Z\subset V(G)$ the following inequality holds:
$$m_Z(d) \leq \ud(Z)\leq M_Z(d).$$
\end{defn}

\begin{prop}\cite[Proposition 4.1]{Caporasocompactification}
    Let $G$ be a semistable graph. Then any divisor class $\delta \in \Pic(G)$ contains a semibalanced representative.
\end{prop}

\begin{rmk}
    Semibalanced representatives are not necessarily unique in their equivalence class $\delta$. They are unique in every class $\delta$ of degree $d$ if
    $G$ is stable and $d-g+1$ and $2g-2$ are coprime.

    In the literature, semibalanced divisors are sometimes also called semistable. Furthermore, semibalanced divisors in degree $d = g - 1$ are exactly the orientable divisors; semibalanced divisors in degree $g$ can be described in terms of generalized orientations and are precisely the so-called break divisors (see \cite{CC19, CPS} for further details). 
\end{rmk}

\medskip

A second class of divisors, reduced divisors, has featured prominently in the study of the Baker-Norine rank, which will be discussed below.

\begin{defn} \label{def:reduced} Let $\ud$ be a divisor on a 
graph $G$ and fix a subset of vertices $V \in  V(G)$. We say that $\ud$ is $V$-\textit{reduced} if
\begin{enumerate}
\item $\ud(v)\geq 0$, for all $v \in V(G)\setminus V$;
\item for every non-empty set $A \subset V(G)\setminus V$, there exists a vertex $v\in A$ such that $\ud(v)< v \cdot A^{c}$.
\end{enumerate}
\end{defn}

If $V = \{u\}$ consists of a single vertex, we will call $\{u\}$-reduced divisors just $u$-reduced divisors. This is the classical case, the above generalization was introduced in \cite{B21, C23} independently by the first and second authors. Notice also that even if the original definitions were made for graphs with no loops or weights, both the definition and the proposition below immediately generalize to arbitrary graphs.

\begin{prop}\cite[Proposition 3.1]{BN} \label{exiured} Let $G$ be a 
graph and fix a vertex $u\in V(G)$. Then for every divisor $\ud$ of $G$ there exists a unique $u$-reduced divisor $\ud'\in \Div(G)$ in the equivalence class of $\ud$.
\end{prop}

\subsection{The Baker-Norine combinatorial rank} Now we discuss the notion of the purely combinatorial rank that  was introduced by
Baker and Norine in \cite{BN}.
Let $G$ be a loopless and weightless graph and $\ud\in\Div(G)$ a divisor. The \textit{Baker-Norine rank} of $\ud$ is defined by
$$r_G(\ud)=\max \{k:\forall \underline{e}\in \Div_{+}^k(G), \ \ \exists \ \ \underline{d'}\sim \ud \mbox{ such that } \underline{d'}-\underline{e}\ge 0 \}$$
with $r_G(\ud)=-1$ if the set is empty.

Let $G$ be a graph possibly with weights or loops, and consider the weightless and loopless graph $G^{\bullet}$ defined by attaching at each vertex $v\in V(G)$ exactly $\omega(v)$ loops based at $v$, and then adding, at each loop, one new vertex subdividing the edge.
For each $\ud \in \Div(G)$,  we can naturally  define $\ud^{\bullet}\in \Div(G^{\bullet})$ by setting $\ud^{\bullet}$ to be $0$ at the new vertices of $G^\bullet$.
In this way we get a natural injective homomorphism $\iota: \Div(G) \to \Div(G^{\bullet})$ inducing an injective homomorphism $\Pic(G)\hookrightarrow\Pic(G^{\bullet})$.
The definition of the Baker-Norine rank has been extended in \cite{AC13} to arbitrary graphs $G$ by setting for any divisor $\ud \in \Div(G)$: $$r_G(\ud):= r_{G^{\bullet}}(\iota(\ud)).$$

Notice that
the Baker-Norine rank is constant in an equivalence class by definition. We therefore write $r_G(\delta) = r_G(\ud)$ if $\ud$ is a representative of the divisor class $\delta = [\ud]$.

Baker and Norine \cite{BN} also proved that this rank on weightless and loopless graphs satisfies the Riemann-Roch theorem. Amini and Caporaso \cite{AC13} extended this result to arbitrary graphs.

\begin{thm}[Riemann-Roch Theorem for graphs]\cite{BN, AC13}\label{rrg} Let $G$ be a
graph of genus $g$, possibly with weights and loops, and $\ud\in \Div(G)$ a divisor of degree $d$. We have
	$$r_G(\ud)-r_G(\underline{k}_G - \ud)=d-g+1.$$	
\end{thm}

Furthermore, it is easy to see that for any two divisors $\ud, \ud' \in \Div(G)$ one has \[r_G(\ud) + r_G(\ud') \leq r_G(\ud + \ud').\] This, together with the Riemann-Roch theorem for graphs, immediately implies the Clifford inequality for the Baker-Norine rank:

\begin{cor}[Clifford inequality]\cite[Corollary 3.5]{BN} \label{cor:cliffBN}
    For any divisor class $\delta$ of degree $0 \leq d \leq 2g -2$, we have \[r_G(\delta) \leq \frac{d}{2}.\] 
\end{cor}

The following basic properties of the rank function are further consequences of the Riemann-Roch Theorem for graphs.

\begin{cor}\cite{C1}\label{corrrfg}  Let $G$ be a graph of genus $g$ and let $\ud \in \Div^d(G)$. Then:
\begin{itemize}
\item[(a)] If $d=0$, then $r_G(\ud) \leq 0$, and equality holds if and only if $\ud \sim \un{0}$.
\item[(b)] If $d=2g-2$, then $r_G(\ud) \leq g-1$ and equality holds if and only if $\ud \sim \un{k}_G$.
\item[(c)] If $d<0$, then $r_G(\ud) = -1$.
\item[(d)] If $d> 2g-2$, then $r_G(\ud) = d-g$.
\end{itemize}
\end{cor}

\subsection{The algebraic rank}

Next, we consider a nodal curve $X$ and a line bundle $L$ on $X$. We can associate to the pair $(X,L)$ its combinatorial type, which is  the pair $(G_X, \underline{\deg}(L))$ where $G_X$ is the dual graph of $X$ and $\underline{\deg}(L)$ is the \textit{multidegree} of $L$.
Recall that $G_X$ is a weighted graph that has a vertex $v$ for each irreducible component $C_v$ of $X$, an edge for each node, and weight function $\omega$ given by associating to a vertex the geometric genus of the corresponding irreducible component of $X$. The multidegree $\underline{\deg}(L)$, on the other hand, is the divisor on $G_X$ whose value on a vertex is given by the degree of the restriction of $L$ to the corresponding irreducible component of $X$. 

We denote by $M_G$ the set of isomorphism classes of curves having $G$ as dual graph.
Given a divisor $\ud= (d_1, \ldots, d_{\lambda}) \in \ZZ^{\lambda}$, we set as usual
\[\Pic^{\ud}(X):= \{L\in \Pic (X) : \underline{\deg}(L)=\ud \},\]
the variety of isomorphism classes of line bundles of multidegree $\ud$. The varieties $\Pic^{\ud}(X)$ are the connected components of the degree $d$ Picard variety $\Pic^d(X)$. For every curve $X\in M_G$, we have $\underline{\deg} (K_X)=\underline{k}_{G_X}$, where $K_X$ denotes the dualizing bundle on $X$.



In \cite{C1}, Caporaso introduced a way to give an algebraic interpretation for the Baker-Norine rank of $\underline{\deg}(L)$ on the nodal curves $X$ with dual graph $G$ themselves, i.e., without choosing a smoothing of $X$ as in \cite{B}. More precisely, she defined the so-called \textit{algebraic rank} $\ralg$, as follows:
First, we set \[r^{\mathrm{max}}(X, \md) = \max \left \{r(X, L) \mid L \in \Pic^{\md}(X) \right \},\]
where $r(X,L) = h^0(X,L) - 1$ denotes the rank of the line bundle $L$.
Then 
\[r^{\min}(X,\delta):=\min \left\{\rmax(X,\ud) \mid \forall \ud\in \delta \right\},\]
and finally 
\[\ralg(G, \delta) := \max \left \{r^{\min}(X,\delta) \mid X \in M_G \right \}. \]

That is, we have:

\begin{equation}
\ralg(G,\delta):=\max_{X\in M_G } \left\{ \min_{\ud \in \delta} \left \{ \max_{L\in \Pic^{\ud}(X)} \left \{r(X,L) \right\} \right\}\right \}.\label{ralg2}
\end{equation}
In many examples, the algebraic rank coincides with the Baker-Norine rank but there are examples where the two differ \cite{CLM, Len17}. However, Caporaso, Len, and the third author were able to show in \cite{CLM} that the Baker-Norine rank is an upper bound for the algebraic rank of divisors on finite graphs. We can summarize what is known as follows:
\begin{fact}\label{fact:eqrgralg} Let $G$ be a finite graph of genus $g$, $\md$ a divisor with class $\delta$ of degree $d$, and $X \in M_G$ a curve with dual graph $G$. 
\begin{itemize}
    \item[(1)] \cite[Proposition 2.6]{CLM}  Riemann-Roch holds for $\rmax, r^{\min},$ and $\ralg$:
    \begin{itemize}
		\item[(a)] $\rmax (X, \ud) - \rmax (X, \underline{k}_G - \ud )= d - g +1;$
		\item[(b)] $r^{\min} (X , \delta) - r^{\min}  (X , [\underline{k}_G - \ud])= d - g +1;$
		\item[(c)] $\ralg (G, \delta) - \ralg (G,  [\underline{k}_G - \ud])= d - g +1.$
    \end{itemize}
    \item[(2)] \cite[Theorem 4.2]{CLM} \label{allcases}
	 The algebraic rank is bounded by the Baker-Norine rank:
	\begin{equation}\label{des}
	\ralg(G, \delta)\leq r_G(\delta).
	\end{equation}
    \item[(3)]\cite[Proposition 4.6]{CLM} If $0\leq d \leq 2g-2$, then the Clifford inequality holds: $$\displaystyle \ralg (G,\delta) \leq \frac{d}{2}.$$ 
    \item[(4)]\cite[Theorem $2.9$]{C1}\label{teo29} If $d\geq 2g -2$, $G$ is semistable and $\ud$ is semibalanced, then $\rmax(X,\ud) = r_G(\ud)$.
    \item[(5)] Equality $\ralg(G,\delta) = r_G(\delta)$ has been established in the following cases:
    \begin{itemize}
        \item [(a)] \cite[Theorem 2.9]{C1} $d \geq 2g - 2$ or $d \leq 0$.
        \item [(b)] \cite[Theorem 1.2]{KY16} $g \leq 3$ and $G$ is not hyperelliptic. 
        \item [(c)] \cite[Theorem 1.1]{KY16} $G$ is hyperelliptic and the base field does not have characteristic $2$. 
        \item[(d)] \cite[Corollary 2.11]{C1} $G$ has only one vertex.
        \item [(e)] \cite[Proposition 5.6]{CLM} $G$ is weightless and loopless with two vertices (a binary graph). 
        \item [(f)] \cite[Theorem 5.13]{CLM} $G$ is weightless and loopless and $\delta$ is rank-explicit, i.e., $\ud$ is $u$-reduced for some vertex $u$ and $\ud(u)$ is either negative or minimal among the values $\ud(v), v \in V(G)$. 
    \end{itemize}
\end{itemize}
\end{fact}

\section{The uniform algebraic rank}
\label{sec:uniform algebraic rank}

In this section, we introduce a new version of the algebraic rank, in which we first vary the curve in $M_G$ and only then vary $\md \in \delta$. Namely,
given a graph $G$ and a divisor $\ud\in \Div(G)$, we define
$$\rmaxx(G, \ud):=\max\{\rmax(X,\ud) \mid X\in M_G\},$$
and for any $\delta\in \Pic^d(G)$, we define the \emph{uniform algebraic rank}:
\begin{equation}\label{rALG}
\ralgg(G,\delta):=\min \left \{\rmaxx(G, \ud) \mid \ud \in \delta \right  \}.
\end{equation}

That is, we have 
\begin{equation} \label{eq:uniform rank}
\ralgg(G,\delta):=\min_{\ud \in \delta} \left \{\max_{X\in M_G} \left \{ \max_{ L \in \Pic^{\ud}(X)} \left \{ r(X,L) \right \} \right \} \right \}.
\end{equation}

\begin{rmk}
    Compared to the definition of $\ralg$, the order of varying $X \in M_G$ and $\md \in \delta$ is switched in the definition of $\ralgg$ (see \eqref{ralg2} vs. \eqref{eq:uniform rank}).
    On the other hand, this is the only possible variation of the definition of algebraic rank given by switching the order in which the objects are varied: in order to define $L$ it is necessary to first fix both $X$ and $\md$. 
\end{rmk}

In this section, we will compare this notion of uniform algebraic rank to related definitions.

\subsection{Comparison to the algebraic rank}

We begin our discussion by observing that the uniform algebraic rank is an upper bound for the algebraic rank:

\begin{prop}\label{prop:ineqralgalgg} Let $G$ be a graph and let $\delta\in \Pic(G)$, then
\[ \ralg(G,\delta) \leq \ralgg(G,\delta). \]
\end{prop}

\begin{proof}  Suppose that $\ralg(G,\delta)$ is realized by $(X_1,\ud_1)$ , i.e.,
\[\ralg(G,\delta) =r^{\min}(X_1,\delta) = \rmax(X_1,\ud_1)\]
 and suppose also that $\ralgg(G,\delta)$ is realized by $(X_2,\ud_2)$, i.e.,
 \[\ralgg(G,\delta) = \rmaxx(G,\ud_2) = \rmax(X_2,\ud_2),\]
where $X_1, X_2\in M_G$ and $\ud_1,\ud_2 \in [\ud]$. So, since $\ralg(G,\delta) = \rmax(X_1,\ud_1)$, we have $ \rmax(X_1, \ud_1) \leq \rmax(X_1, \ud')$,  for all $\ud' \in [\ud]$. In particular,
\begin{equation}\label{rmrm}
\rmax(X_1, \ud_1) \leq \rmax(X_1, \ud_2).
\end{equation}
Since $\ralgg(G,\ud) = \rmax(X_2,\ud_2)$ we have $\rmax(X_2,\ud_2) \geq \rmax(Y, \ud_2)$, for all \linebreak $Y\in M_G$. In particular,
\begin{equation}\label{rmrmm}
\rmax(X_2, \ud_2) \geq \rmax(X_1,\ud_2 ).
\end{equation}
By (\ref{rmrm}) and (\ref{rmrmm}) we have
\[  \rmax(X_1, \ud_1) \leq \rmax(X_1, \ud_2) \leq \rmax(X_2, \ud_2) .\]
Therefore,
\[ \ralg(G,\ud) \leq \ralgg(G,\ud).\]
\end{proof}

\begin{rmk} \label{remdefunialgr}
    Another way to see that the above mentioned inequality holds is as follows. Fix $s \in \mathbb Z$. Then:
    \begin{enumerate}
        \item We have $\ralg(G, \delta) \leq s$ if and only if $\forall X \in M_G, \exists \md \in \delta, \forall L \in \Pic^{\ud}(X): r(X,L) \leq s$. 
        \item We have $\ralgg(G, \delta) \leq s$ if and only if $\exists \md \in \delta, \forall X \in M_G, \forall L \in \Pic^{\ud}(X): r(X,L) \leq s$. 
    \end{enumerate}
    Clearly, the second statement implies the first, that is, if $\ralgg(G, \delta) \leq s$ then also $\ralg(G, \delta) \leq s$. This implies the inequality of Proposition~\ref{prop:ineqralgalgg}.
\end{rmk}

The next example shows that the uniform algebraic rank can be larger than the algebraic rank:

\begin{ex} \label{ex:strict ineq}
    Let $G$ have three vertices $v_1, v_2, v_3$, with three edges between $v_1$ and $v_2$, and one edge between $v_2$ and $v_3$ as in Figure~\ref{fig3}. The weights of the vertices are $0, 3, 1$, respectively. 

    \tikzset{every picture/.style={line width=0.75pt}}
    \begin{figure}[ht]	
    \begin{tikzpicture}[x=0.6pt,y=0.6pt,yscale=-0.8,xscale=0.8]
    \import{./}{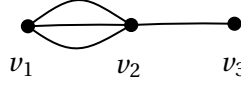}
    \end{tikzpicture}		
    \caption{The graph in Example~\ref{ex:strict ineq}.}
    \label{fig3}
    \end{figure}
    
    Denote by $C_i$ the irreducible component corresponding to $v_i$ of a curve $X$ with dual graph $G$. In particular, $C_2$ is a smooth curve of genus $3$, so it is either trigonal or hyperelliptic, but not both (here we mean by trigonal that it admits a \emph{base point free} $g_3^1$). 

    Set $\md = (0, 3, 2)$ and $\md' = (3,2,0)$ and observe that $\md\sim\md'$. We determine $r^{\max}$ for these two divisors, depending on whether $C_2$ is hyperelliptic or not. Throughout we use the results of \cite[Lemmas 1.4 and 1.5]{caporasosemistable} describing how gluing along nodes affects the dimension of the space of global sections. 
 
    Consider first a line bundle $L$ of multidegree $\md = (0, 3, 2)$. Let $X'$ be the partial normalization of $X$ with dual graph given by deleting two edges between $v_1$ and $v_2$. Since $L|_{C_1}$ and $L|_{C_3}$ are base point free, we have for the pull-back $L'$ of $L$ to $X'$ that $h^0(X', L') = h^0(C_2, L|_{C_2}) + 1$. Thus if $h^0(C_2, L|_{C_2}) \leq 1$ we have $h^0(X,L) \leq 2$. So assume $h^0(C_2, L|_{C_2}) = 2$, which is the maximal possible value by Clifford's inequality. 

    \begin{enumerate}
        \item If $C_2$ is hyperelliptic, then $L|_{C_2}$ equals a copy of the $g_2^1$ plus a base point $r$. If $r \in C_1 \cap C_2$, all global sections of $L$ need to vanish along $C_1$. If $r \not \in C_1 \cap C_2$, any global section of $L$ that vanishes along $C_1$ needs to vanish also along $C_2$. Both statements are not true for $L'$ and thus in either case not all sections of $L'$ descend to $L$. Hence $h^0(X,L) \leq 2$ and $r^{\max}(X, \md) =  1$ if $C_2$ is hyperelliptic. 
        \item If $C_2$ is trigonal, it is not hyperelliptic and thus $L|_{C_2}$ is a base point free $g_3^1$. If $C_1$ and $C_2$ are glued on $C_2$ along a divisor in this $g_3^1$, all global sections of $L'$ descend to global sections of $L$ for an appropriate choice of gluing over the nodes normalized in $X'$. For such a choice we will have $h^0(X,L) = 3$ and $r^{\max}(X, \md) =  2$. 
    \end{enumerate}

    Next, let us consider the case of a line bundle $L$ of multidegree $\md' = (3,2,0)$. A similar calculation as for $\md$ yields the following: we may assume that $L|_{C_3} = \cO_{C_3}$. With notation as in the case for $\md$ above, we get $h^0(X',L') = 3 + h^0(C_2, L|_{C_2})$. Since subtracting any two points from $L|_{C_1} \simeq \cO_{\PP^1}(3)$ gives a base point free linear system, we need to have that $h^0(X,L) = 1 + h^0(C_2, L|_{C_2})$. It follows that $r^{\max}(X, \md') =  2$ if $C_2$ is hyperelliptic and $r^{\max}(X, \md') =  1$ if $C_2$ is trigonal.

    In summary, we obtain:
    
    \begin{center}
    \begin{tabular}{| c || c | c | } \hline 
    \diagbox[
    innerwidth=2.2cm,innerleftsep=0.8cm]{$\delta$}{$M_G$} 
     & $X$ with $C_2$ hyperelliptic & $X^*$ with $C_2$ trigonal\\
  \hline \hline
 $\ud = (0,3,2)$                      & $\rmax(X,\ud)=1$ & $\rmax(X^*,\ud)=2$ \\
 \hline
 $\ud' = (3,2,0)$                      & $\rmax(X,\ud')=2$ & $\rmax(X^*,\ud')=1$ \\
 \hline
\end{tabular}
\end{center}

Thus $r^{\min}(X,\delta) \leq 1$ and $r^{\min}(X^*,\delta) \leq 1$. Since any curve with dual graph $G$ is either of the form $X$ or $X^*$ (that is, has either $C_2$ hyperelliptic or trigonal), it follows that  \[\ralg(G,\delta) \leq 1.\] 

On the other hand, $\rmaxx(G, \md) = \rmaxx(G, \md') = 2.$ To show that we indeed have \[\ralgg(G,\delta) \geq 2, \]
it suffices to produce for every $\md'' \in \delta$ a curve $X$ with dual graph $G$ and a line bundle $L$ on $X$ of multidegree $\md''$ and rank at least $2$. By Theorem~\ref{thm:effective realization} below, it suffices to do so for all effective $\md'' \in \delta$. We already checked this for $\md$ and $\md'$, and we check the remaining cases below. Whenever the multidegree has value $0$ on $C_3$ we set $L|_{C_3} = \mathcal O_{C_3}$. 

\begin{enumerate}[(a)]
    \item If $\md'' = (0,5,0)$ the restriction $L|_{C_2}$ has degree $5$ on a genus $3$ curve, hence $h^0(C_2, L|_{C_2}) = 3$. Let $C_2$ be hyperelliptic and  $L|_{C_2} = g_2^1 + p + q + r$ with no two points in $p,q,r$ conjugate under the hyperelliptic involution and such that $C_1 \cap C_2 = \{p,q,r\}$. This implies that subtracting one of the points $p,q,r$ from $L|_{C_2}$ turns the other two into base points. With notation as above, we have $h^0(X', L') = 3$ and by the choice of $L|_{C_2}$ there is a choice of gluing data along $C_1 \cap C_2$ such that all global sections of $L'$ descend to $L$ and hence $h^0(X,L) = 3$, as well.
    \item If $\md'' = (0,4,1)$, choose $C_2$ trigonal, $L|_{C_2} = g_3^1 + r$ where $r = C_2 \cap C_3$, and $L|_{C_3} = \cO_{C_3}(r)$. Then $h^0(X', L') = 3$. If the points in $C_1 \cap C_2$ are chosen so that they form the $g_3^1$ on $C_2$, there is a gluing along the nodes in $C_1 \cap C_2$ such that all global sections of $L'$ descend to $L$. For this choice, we obtain $h^0(X,L) = 3$.
    \item If $\md'' = (0,2,3)$, let $C_2$ be hyperelliptic, $L|_{C_2} = g_2^1$, and $C_1 \cap C_2$ contain two points that are conjugate under the hyperelliptic involution on $C_2$. Then $h^0(X', L') = 4$ and there is gluing data over the nodes of $C_1 \cap C_2$ that imposes only one condition on global sections in passing from $L'$ to $L$. For this choice we obtain  $h^0(X,L) = 3$. 
    \item If $\md'' = (0,1,4)$ or $\md'' = (0,0,5)$, $h^0(C_3, L|_{C_3}) \geq 4$ and hence the space of global sections of $L$ vanishing along $C_1 \cup C_2$ already has dimension at least $3$, and hence so does $h^0(X,L)$.
    \item If $\md'' = (3,1,1)$, set $r = C_2 \cap C_3$ and $L|_{C_2} = \cO_{C_2}(r)$ as well as $L|_{C_3} = \cO_{C_3}(r)$. Then $h^0(X', L') = 5$ and gluing along the two normalized nodes in $C_1 \cap C_2$ can impose at most two conditions. Hence $h^0(X,L) \geq 3$. 
    \item Finally, if $\md'' = (3,0,2)$ then $h^0(X', L') = 5$ and gluing along the two normalized nodes in $C_1 \cap C_2$ can impose at most two conditions. Hence again $h^0(X,L) \geq 3$.
\end{enumerate}
\end{ex}

\begin{rmk}
    One case in which we do have $\ralgg(G,\delta) = \ralg(G, \delta)$ is if $G$ is weightless and each vertex has valence at most $3$. In this case there is, up to isomorphism, a unique curve $X$ with dual graph $G$. 
\end{rmk}

\subsection{Comparison to the Baker-Norine rank}
The maybe most consequential property of the algebraic rank $\ralg$ is that it is bounded from above by the Baker-Norine rank  \cite{CLM}. In this section, we show that the same holds for the uniform algebraic rank $\ralgg$.

Before we can prove this comparison result, we need some preliminaries.
For each $v\in V(G)$ set
\[g(v):=\omega(v)+l(v),\]
where $l (v)$ is the number of loops adjacent to $v$.
\begin{defn}\label{edegdef} Let $G$ be a graph and let $\underline{e}$ be an effective divisor of $G$. We define the effective divisor $\underline{e}^{\deg}$ on $G$ so that for every $v\in V$ we have
\[ \underline{e}^{\deg} (v)= \underline{e}(v) + \min \{ \underline{e}(v), g(v)\}. \]
\end{defn}
In particular, if $G$ is a weightless and loopless graph then $\underline{e}^{\deg} = \underline{e}$. 

\begin{lma}\cite[Lemma 3.3]{CLM}\label{rdgeqr} Let $G$ be a graph and let $\ud\in \Div(G)$. If for every effective divisor $\un{e}$ of degree $s$ the divisor $\ud - \un{e}^{\deg}$ is equivalent to an effective divisor, then $r_G(\ud) \geq s$.
\end{lma}

\begin{thm}\label{thm:mainineq}
    For any graph $G$ and divisor class $\delta$ on $G$ we have \[r^{\mathrm{ALG}}(G,\delta) \leq r_G(\delta).\]  
\end{thm}

\begin{proof}
Suppose that $\ralgg(G,\delta) =s$, we want to prove that $r_G(\delta) \geq s$. Since $r_G(\ud)\geq -1$, we can assume $s\geq 0$.
By Lemma~\ref{rdgeqr} it is enough to prove that for all $\underline{e}\in \Div^s_{+}(G)$, there exists $\ud \in \delta$ such that $\ud - \underline{e}^{\deg} \geq 0$. 

By Proposition~\ref{exiured}, there exists $\ud\in \delta$ such that $\ud - \un{e}^{\deg}$ is $u$-reduced for $u\in V(G)$. 
By definition, this implies that $\ud - \un{e}^{\deg}$ is effective away from $u$ and it remains to show that $(\ud - \un{e}^{\deg})(u)\geq 0$. 

Observe that if $\ralgg(G, \delta)=s$, then for all $\ud'\in \delta$, $\rmaxx(G, \ud')\geq s$. In particular, if we consider the representative $\ud$ of $\delta$ such that $\ud - \un{e}^{\deg}$ is $u$-reduced as above, then there exist a curve $X\in M_G$ and a line bundle $L\in \Pic^{\ud}(X)$ such that $r(X,L)\geq s$. By the proof of \cite[Theorem $4.2$]{CLM}, the existence of such a line bundle is enough to conclude that $(\ud - \un{e}^{\deg})(u)\geq 0$, and therefore $r_G(\ud)\geq s$.
\end{proof}

\begin{rmk} \label{rmk:strict ineq}
    Examples where the inequality $\ralg (G,\delta) \leq r_G(\delta)$ is strict can be found in \cite[Examples 5.15 \& 5.16]{CLM}, the latter one was strengthened for metrized complexes by Len \cite{Len17}. Both examples  can be extended for the uniform rank using Remark \ref{remdefunialgr}, that is, $\ralgg (G,\delta) < r_G(\delta)$ also in these cases. In both these examples $r_G(\delta)=2$, and we don't know any example where $r_G(\delta)=1$ and for which $\ralg (G,\delta) < r_G(\delta)$. 
\end{rmk}

\begin{cor}[Clifford inequality] \label{cor:CliffALG}
    For any divisor class $\delta$ of degree $0 \leq d \leq 2g -2$, we have \[r^{\mathrm{ALG}}(G, \delta) \leq \frac{d}{2}.\] 
\end{cor}
\begin{proof}
    This follows directly from Theorem \ref{thm:mainineq} and the Clifford inequality for the Baker-Norine rank, Corollary~\ref{cor:cliffBN}.
\end{proof}

\begin{rmk}
    The inequalities $\ralg(G, \delta) \leq \ralgg(G, \delta) \leq r_G(\delta)$ ensure that whenever $\ralg(G, \delta) = r_G(\delta)$, also $\ralg(G, \delta) = \ralgg(G, \delta) = r_G(\delta)$. This happens, in particular, in the cases summarized in Fact~\ref{fact:eqrgralg} (5).
\end{rmk}

\subsection{Specialization of ranks}

The usefulness of the Baker-Norine rank $r_G(\delta)$ in algebraic geometry mainly comes from Baker's specialization lemma \cite{B}, which relates the rank of a line bundle on a smooth curve to the Baker-Norine rank of the multidegree of its specialization to a nodal central fiber. In the case of the algebraic rank, this is in fact an immediate consequence of upper-semicontinuity of the algebro-geometric rank, as observed in \cite[Lemma 2.7]{CLM}. The inequality $\ralg(G, \delta) \leq \ralgg(G, \delta)$ of the previous section implies that the same is true for the uniform algebraic rank (alternatively, this again follows from upper-semicontinuity of the algebro-geometric rank). 

More precisely, let $\mathcal X \to \Spec R$ be a \emph{regular one-parameter smoothing} of a nodal curve $X$; that is, a flat family of curves over a discrete valuation ring $R$ with smooth generic  fiber $\mathcal X_\eta$, special fiber $X_0$ isomorphic to $X$ and smooth total space. Then any line bundle $\mathcal L_\eta$ on $\cX_\eta$ extends to a line bundle $\mathcal L$ with central fiber a line bundle $L$ on $X$ whose multidegree we denote by $\ud$. The extension $\mathcal L$ is not unique, since twisting by components of $X$ gives non-isomorphic extensions of $\mathcal L_\eta$; but the multidegree of any two extensions lies in the same class $\delta = [\ud]$. In this situation we have:


\begin{cor}[Specialization]
    Let $X$ be a connected curve with dual graph $G$. Let $\mathcal{X}\to \Spec(R)$ be a regular one-parameter smoothing of $X$. Let $\mathcal L$ be a line bundle on $\mathcal X$ that restricts to a line bundle $\mathcal L_\eta$ on the generic fiber $\mathcal X_\eta$ and denote by $\delta$ the class of the multidegree of the restriction of $\mathcal L$ to the central fiber. Then 
\[r(\mathcal X_\eta, \mathcal{L}_\eta) \leq \ralgg(G, \delta).\]
\end{cor}

\begin{proof}
    By \cite[Lemma 2.7]{CLM} we have $r(\mathcal X_\eta, \mathcal{L}_\eta) \leq \ralg(G, \delta)$ and by Theorem~\ref{thm:mainineq} we have $\ralg(G, \delta) \leq \ralgg(G, \delta)$.
\end{proof}

\subsection{Other notions of rank} \label{sec:other ranks}

Summarizing, and using the same notation as above, we thus have 
\[r(\mathcal X_\eta, \mathcal{L}_\eta) \leq \ralg(G, \delta) \leq \ralgg(G, \delta) \leq r_G(\delta),\]
where all inequalities can be strict. The first by \cite[Corollary 3.3]{Len17}, the second by Example \ref{ex:strict ineq} and for the third cf. Remark~\ref{rmk:strict ineq}. Note, however, that in all examples mentioned only one of the above inequalities is strict. We do not know whether it is possible to find examples in which (at least) two inequalities are strict at the same time.

Since the introduction of the Baker-Norine rank and Baker's specialization lemma in \cite{BN, B}, various other efforts have been made to characterize the gap in the inequality $r(\mathcal X_\eta, \mathcal{L}_\eta) \leq r_G(\delta)$. An overview of what is known and how the uniform algebraic rank fits in this picture is sketched in Figure~\ref{fig1}. 

\tikzset{every picture/.style={line width=0.75pt}}
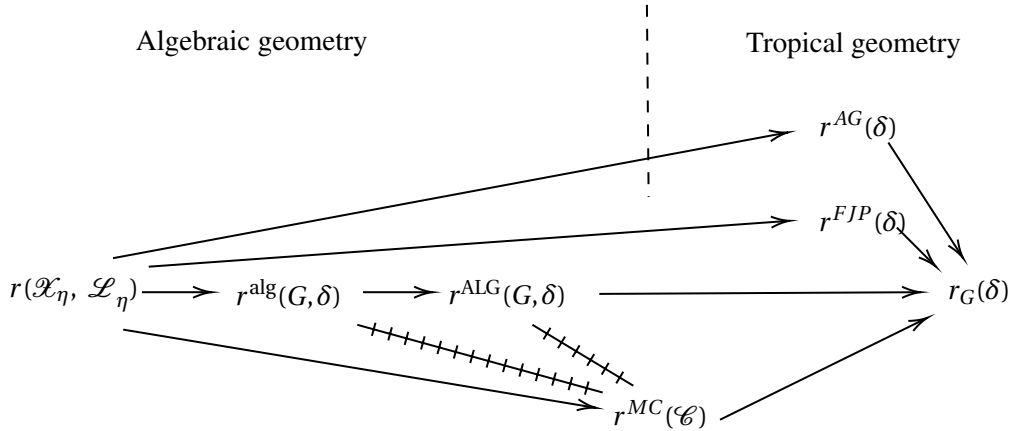
\begin{figure}[ht]	
\begin{tikzpicture}[x=0.8pt,y=0.8pt,yscale=-0.8,xscale=0.8]
\draw  [dash pattern={on 4.5pt off 4.5pt}]  (439.33,10.83) -- (440.33,123.83) ;
\draw    (142.33,179.83) -- (181.33,179.83) ;
\draw [shift={(183.33,179.83)}, rotate = 180] [color={rgb, 255:red, 0; green, 0; blue, 0 }  ][line width=0.75]    (10.93,-3.29) .. controls (6.95,-1.4) and (3.31,-0.3) .. (0,0) .. controls (3.31,0.3) and (6.95,1.4) .. (10.93,3.29)   ;
\draw    (272.33,179.83) -- (311.33,179.83) ;
\draw [shift={(313.33,179.83)}, rotate = 180] [color={rgb, 255:red, 0; green, 0; blue, 0 }  ][line width=0.75]    (10.93,-3.29) .. controls (6.95,-1.4) and (3.31,-0.3) .. (0,0) .. controls (3.31,0.3) and (6.95,1.4) .. (10.93,3.29)   ;
\draw    (411.33,180.83) -- (598.33,179.84) ;
\draw [shift={(600.33,179.83)}, rotate = 179.7] [color={rgb, 255:red, 0; green, 0; blue, 0 }  ][line width=0.75]    (10.93,-3.29) .. controls (6.95,-1.4) and (3.31,-0.3) .. (0,0) .. controls (3.31,0.3) and (6.95,1.4) .. (10.93,3.29)   ;
\draw    (146.33,164.83) -- (521.34,137.98) ;
\draw [shift={(523.33,137.83)}, rotate = 175.9] [color={rgb, 255:red, 0; green, 0; blue, 0 }  ][line width=0.75]    (10.93,-3.29) .. controls (6.95,-1.4) and (3.31,-0.3) .. (0,0) .. controls (3.31,0.3) and (6.95,1.4) .. (10.93,3.29)   ;
\draw    (125.33,159.83) -- (520.37,86.2) ;
\draw [shift={(522.33,85.83)}, rotate = 169.44] [color={rgb, 255:red, 0; green, 0; blue, 0 }  ][line width=0.75]    (10.93,-3.29) .. controls (6.95,-1.4) and (3.31,-0.3) .. (0,0) .. controls (3.31,0.3) and (6.95,1.4) .. (10.93,3.29)   ;
\draw    (131.33,201.83) -- (407.36,247.51) ;
\draw [shift={(409.33,247.83)}, rotate = 189.4] [color={rgb, 255:red, 0; green, 0; blue, 0 }  ][line width=0.75]    (10.93,-3.29) .. controls (6.95,-1.4) and (3.31,-0.3) .. (0,0) .. controls (3.31,0.3) and (6.95,1.4) .. (10.93,3.29)   ;
\draw    (482.33,254.83) -- (602.54,195.72) ;
\draw [shift={(604.33,194.83)}, rotate = 153.81] [color={rgb, 255:red, 0; green, 0; blue, 0 }  ][line width=0.75]    (10.93,-3.29) .. controls (6.95,-1.4) and (3.31,-0.3) .. (0,0) .. controls (3.31,0.3) and (6.95,1.4) .. (10.93,3.29)   ;
\draw    (586.33,141.83) -- (607.92,163.42) ;
\draw [shift={(609.33,164.83)}, rotate = 225] [color={rgb, 255:red, 0; green, 0; blue, 0 }  ][line width=0.75]    (10.93,-3.29) .. controls (6.95,-1.4) and (3.31,-0.3) .. (0,0) .. controls (3.31,0.3) and (6.95,1.4) .. (10.93,3.29)   ;
\draw    (581.33,91.83) -- (624.25,158.15) ;
\draw [shift={(625.33,159.83)}, rotate = 237.09] [color={rgb, 255:red, 0; green, 0; blue, 0 }  ][line width=0.75]    (10.93,-3.29) .. controls (6.95,-1.4) and (3.31,-0.3) .. (0,0) .. controls (3.31,0.3) and (6.95,1.4) .. (10.93,3.29)   ;
\draw    (269,199) -- (413.33,238.83) (279.7,197.8) -- (277.58,205.52)(289.34,200.46) -- (287.22,208.18)(298.98,203.13) -- (296.85,210.84)(308.62,205.79) -- (306.49,213.5)(318.26,208.45) -- (316.13,216.16)(327.9,211.11) -- (325.77,218.82)(337.54,213.77) -- (335.41,221.48)(347.18,216.43) -- (345.05,224.14)(356.82,219.09) -- (354.69,226.8)(366.46,221.75) -- (364.33,229.46)(376.1,224.41) -- (373.97,232.12)(385.74,227.07) -- (383.61,234.78)(395.38,229.73) -- (393.25,237.44)(405.02,232.39) -- (402.89,240.1) ;
\draw    (372,199) -- (431.33,234.83) (382.63,200.75) -- (378.49,207.59)(391.19,205.92) -- (387.05,212.76)(399.75,211.09) -- (395.61,217.93)(408.31,216.25) -- (404.17,223.1)(416.87,221.42) -- (412.73,228.27)(425.43,226.59) -- (421.29,233.44) ;

\draw (62,169.4) node [anchor=north west][inner sep=0.75pt]    {$r(\mathcal{X_{\eta } ,\ L}_{\eta })$};
\draw (195,170.4) node [anchor=north west][inner sep=0.75pt]    {$\ralg( G,\delta)$};
\draw (320,170.4) node [anchor=north west][inner sep=0.75pt]    {$\ralgg( G,\delta)$};
\draw (614,170.4) node [anchor=north west][inner sep=0.75pt]    {$r_{G}( \delta )$};
\draw (417,241.4) node [anchor=north west][inner sep=0.75pt]    {$r^{MC}(\mathcal{C})$};
\draw (538,127.4) node [anchor=north west][inner sep=0.75pt]    {$r^{FJP}( \delta )$};
\draw (538,71.4) node [anchor=north west][inner sep=0.75pt]    {$r^{AG}( \delta )$};
\draw (137,24) node [anchor=north west][inner sep=0.75pt]   [align=left] {Algebraic geometry};
\draw (496,25) node [anchor=north west][inner sep=0.75pt]   [align=left] {Tropical geometry};
\end{tikzpicture}		
\caption{Different notions of rank and their known relations.}
\label{fig1}
\end{figure}

In the figure, all entries denote some notion of rank whose value lies between $r(\mathcal X_\eta, \mathcal{L}_\eta)$ and $r_G(\delta)$. Apart from the ones already discussed, these are the ranks of metrized complexes $r^{\mathrm{MC}}(\mathcal C)$ of Amini and Baker \cite{AB}, the refinement of the Baker-Norine rank $r^{\mathrm{FJP}}(\delta)$ given by imposing tropical independence due to Farkas, Jensen and Payne \cite[Definition 6.5]{FJP}, and the rank $r^{\mathrm{AG}}(\delta)$ of tropical limit linear series developed by Amini and Gierczak \cite[\S 1.5]{AG22}.\footnote{Part of the definition of both $r^{\mathrm{AG}}(\delta)$ and $r^{\mathrm{FJP}}(\delta)$ is the condition imposed for the Baker-Norine rank, in addition to further assumptions. Thus a complete tropical linear series of rank $r$, in their sense, has Baker-Norine rank at least $r$ and $r^{\mathrm{AG}}(\delta), r^{\mathrm{FJP}}(\delta) \leq r_G(\delta)$.} Two ranks are connected by an arrow if it is known that one is less or equal than the other, with the one towards which the arrow points being bigger. A crossed-out line indicates that it is known that no inequality holds in either direction (the result that $\ralg(G,\delta)$ and $r^{\mathrm{MC}}(\mathcal C)$ are not comparable is due to \cite[\S 3]{Len17}, and the examples presented there generalize to $\ralgg(G,\delta)$). Finally, if there is no arrow between two entries, it means that, as far as we are aware, no relation is known.


\section{Properties of the uniform algebraic rank}

\subsection{The Riemann-Roch theorem}

Next, we establish the Riemann-Roch theorem for the uniform algebraic rank $\ralgg(G, \delta)$ and the auxiliary notion $\rmaxx(G, \ud)$ used in its definition (see the beginning of Section~\ref{sec:uniform algebraic rank}). 



\begin{thm}[Riemann-Roch]\label{rrminmax} 
Let $G$ be a finite graph of genus $g$, $\ud$ a divisor of degree $d$ on $G$ with equivalence class $\delta$. Then
	\begin{itemize}
	\item[(a)] $\rmaxx(G, \ud) - \rmaxx(G, \un{k}_G -\ud ) = d - g + 1.$
	\item[(b)] $\ralgg(G,\delta) - \ralgg(G, [\un{k}_G - \ud] ) = d - g + 1.$
	\end{itemize}
\end{thm}
\begin{proof}
Set $\ud^* = \un{k}_G -\ud$ and $\delta^* = [\un{k}_G - \ud]$ (notice that since $\ud \sim \underline{e}$ implies $\ud^* \sim \underline{e}^*$, we have that $\delta^* :=[\ud^*]$ is well defined). 
We follow the arguments of the proof of \cite[Proposition 2.6]{CLM}. By the same proposition we know that
\begin{equation}\label{rrm}
	\rmax(X,\ud)-\rmax(X,\ud^*)=d-g+1.
\end{equation}
We claim that, given $X\in M_G$, we have
\begin{equation}\label{rmaxx} \rmaxx(G,\ud) = \rmax(X,\ud) \Longleftrightarrow \rmaxx(G,\ud^*) = \rmax(X,\ud^*) .
\end{equation}
Observe that \eqref{rrm} and  \eqref{rmaxx} imply that (a) holds. Furthermore, since $\ud^{**}=\ud$, 
it suffices to prove only one  implication in \eqref{rmaxx}.
With this in mind, assume that $\rmaxx(G,\ud)=\rmax(X,\ud)=r(X,L)$, for some $L\in\Pic^{\ud}(X)$. By \eqref{rrm} and Riemann-Roch on $X$, we have $\rmax(X,\ud^*)=r(X,L^*)$, where we use the notation $L^*$ to indicate the residual line bundle $K_X\otimes L^{-1}$.  Suppose by contradiction that there exists a curve $Y\in M_G $ and $M^* \in \Pic^{\ud^*}(Y)$ such that
\[r(Y, M^*)= \rmax(Y,\ud^*)=\rmaxx(G,\ud^*)>r(X, L^*).\]
In this case, by Riemann-Roch on $X$, we have
\[r(X,L)=r(X,L^*)+d-g+1<r(Y,M^*)+d-g+1=r(Y,M).\]
By the definition of $\rmaxx$,
\[ \rmaxx(G, \ud)=\rmax(X,\ud)=r(X,L) \geq \rmax(Y,\ud) \geq r(Y,M),\]
contradicting the previous inequality. Therefore \eqref{rmaxx} is proven.

Now, let $\ralgg(G, \delta)=r(X,L)$. So, \[\rmaxx(G, \ud) = \rmax(X,\ud) = r(X,L)\] 
and, by \eqref{rrm} and \eqref{rmaxx}, we get:
\[\rmaxx(G, \ud^*) = \rmax(X,\ud^*) = r(X,L^*).\]
By Riemann-Roch on $X$, to prove (b) it suffices to prove that $\ralgg(G, \delta^*)=r(X,L^*)$. By contradiction, suppose that there exists $Y\in M_G$, $\underline{e}^*\in \delta^*$ and $N^*\in\Pic^{\underline{e}^*}(Y)$ such that
\[r(X,L^*)>\ralgg(G,\delta^*) = \rmaxx(Y, N^*) = \rmax(Y,\underline{e}^*)=r(Y,N^*).\]
By Riemann-Roch on $X$ we have
\[r(X,L)=r(X,L^*)+d-g+1>r(Y,N^*)+d-g+1=r(Y,N).\]
Since $\underline{e}\in \delta$, it follows that
\[ r(X,L)=\ralgg(G,\delta)\leq \rmaxx(Y,N) = r(Y,N), \]
contradicting the preceding inequality.
\end{proof}

The next corollary shows that, as for the algebraic rank, semibalanced divisors on $G$ with degree outside the special range $0 \leq d \leq 2g - 2$ realize the uniform algebraic rank (see Section~\ref{sec:semibalanced} for the definition of semibalanced). In the special range, it seems to be a difficult question to find explicit representatives $\ud \in \delta$ that realize the uniform algebraic rank; see Section~\ref{sec:explicit clifford} for a related discussion.

\begin{cor} \label{coroteob}  Let $G$ be a semistable graph of genus $g$ and let $\ud \in \Div^d(G)$ be a semibalanced divisor. Then the following facts hold.
\begin{itemize}
\item[(a)] If $d<0$, then $\rmaxx(G,\ud) = r_G(\ud) = -1$.
\item[(b)] If $d>2g-2$, then $\rmaxx(G,\ud) = r_G(\ud) = d - g$.
\item[(c)] If $d=2g-2$, then $\rmaxx(G,\ud) = r_G(\ud) \leq g-1$ and equality holds if and only if $\ud \sim \un{k}_G$.
\item[(d)] If $d=0$, then $\rmaxx(G,\ud) = r_G(\ud) \leq 0$ and equality holds if and only if $\ud \sim \un{0}$.
\end{itemize}
\end{cor}
\begin{proof} 
By \cite[Theorem 2.9]{C1}, if $d \geq 2g-2$ then every semibalanced $\ud\in \delta$ satisfies $\rmax(X,\ud) = r_G(\ud)$, for every $X\in M_G$. Thus $\rmaxx(G,\ud) = r_G(\ud)$ in (b) and (c). Applying Riemann-Roch for $\rmaxx$ (Theorem~\ref{rrminmax}) this implies $\rmaxx(G,\ud) = r_G(\ud)$ also in (a) and (d).  

The remaining statements are well-known for $r_G(\ud)$ (cf. Corollary~\ref{corrrfg}).
\end{proof}

\subsection{Realization by effective representatives}

In this section, we show that to compute both the algebraic rank $\ralg(G, \delta)$ and the uniform algebraic rank $\ralgg(G, \delta)$, it suffices to check effective representatives $\md \in \delta$. As far as we know, this is new already for the algebraic rank. 

To prove this statement, we need to recall the Dhar decomposition of the graph $G$ with respect to a subset of vertices $V$. This is a generalization of the Dhar decomposition with respect to a single vertex $v$ (see, e.g., \cite[\S 3.4]{CLM} for a formulation in our context) studied independently by the first and second authors \cite{B21, C23}.

To this end, let $V \subset V(G)$ be a set of vertices, and $\md$ a divisor on $G$ effective away from $V$. We define a sequence of subsets of vertices \begin{equation} \label{eq:Dhar decomposition}
    V = V_0 \subset V_1 \subset \ldots \subset V_n,
\end{equation}
iteratively as follows. Given $V_i$, to obtain $V_{i+1}$ we add all vertices $v \in V(G) \setminus V_i$ for which $v \cdot V_i > \ud(v)$ (where, as before, we identify $V_i$ with the divisor with value $1$ on each vertex in $V_i$ and $0$ on all other vertices). 

Since there are only finitely many vertices, this process needs to stabilize at some point and we have $V_{n} = V_{n+1}$. We set \[\W(\md, V) \coloneqq V(G) \setminus V_n\] and call \[V(G) = V_n \sqcup \W(\md, V)\] the \emph{Dhar decomposition of $G$ with respect to $V$ and $\md$.} By construction, the divisor $\md - \underline t_{V_n}$ is still effective away from $V$ (and, by definition, linearly equivalent to $\ud$).

Before we can state the main result of this section, we need the following observation. For the definition of $V$-reduced divisors see Definition~\ref{def:reduced}. 

\begin{lma} \label{lma:injectivity}
    Let $X$ be a curve with dual graph $G$ and $Y$ a subcurve of $X$ whose irreducible components correspond to the subset of vertices $V \subset V(G)$. Let $L$ be a line bundle on $X$ whose multidegree $\md$ is $V$-reduced. Then the restriction map \[H^0(X,L) \to H^0(Y,L|_Y)\] is injective. 
\end{lma}

\begin{proof}
    The kernel of the linear map $H^0(X,L) \to H^0(Y,L|_Y)$ is given by global sections $s$ of $L$ that vanish on $Y$. Let $v \not \in V$ be a vertex on which $\md$ turns negative after a chip-firing move along the complement of $V$. This means that $\md(v) < v\cdot V=:k$.
    Since $s$ vanishes by assumption on all components $X_w$ corresponding to vertices $w \in V$, $s$ also vanishes on the $k$ points of intersection of $X_v$ with the components $X_w$. Since $L$ has degree less than $k$ on $X_v$, this implies that $s$ vanishes on all of $X_v$. Repeating this argument shows that $s$ vanishes on all components $X_v$ where $v$ is not in the Dhar set $\W(\md, V)$. However, $\W(\md, V) = \emptyset$ if $\md$ is $V$-reduced, since the vertices $v \in \W(\md, V)$ satisfy by definition $v \cdot \W(\md, V)^c \geq 0$. Thus $s$ needs to vanish on all of $X$, and the kernel of the map $H^0(X,L) \to H^0(Y,L|_Y)$ consists only of the zero section.
\end{proof}

The point of the next theorem is that in order to calculate the algebraic and the uniform algebraic rank, it suffices to restrict to effective divisors $\md \in \delta$.

\begin{thm} \label{thm:effective realization}
    Let $\delta$ be an effective divisor class on a graph $G$.  Then there exists a curve $X$ and an effective divisor $\md$ such that $\ralg(G, \delta) = \rmax(X, \md)$. Similarly, there exists an effective divisor $\md'$ such that $\ralgg(G, \delta) = \rmaxx(G, \md')$.
\end{thm}

\begin{proof}
    For both claims, it suffices to show the following: Suppose $\md$ is not effective, but its class $\delta = [\md]$ is effective. Then there is an effective divisor $\md' \sim \md$ such that for any curve $X$ with dual graph $G$ we have \[\rmax(X, \md') \leq  \rmax(X, \md).\]

    Let $V \subset V(G)$ denote the subset of vertices on which $\md$ fails to be effective, which is non-empty by assumption. Consider the Dhar decomposition $V(G) = V_n \sqcup \W(\md, V)$ with respect to $\md$ and $V$, as described above. By \cite[Proposition 3.8]{C23}, $\W(\md, V)$ is not empty since the class of $\md$ is effective. 
    
    Now let $X$ be a curve with dual graph $G$, $Y \subset X$ the subcurve corresponding to $V_n$ and $Y^c$ the one corresponding to $\W(\md, V)$. Let $L$ be a line bundle of multidegree $\md$ and rank $r$. We construct a line bundle $L'$ of rank at most $r$ and of multidegree \[\md' = \md - \underline t_{V_n}.\] 

    It follows from Lemma~\ref{lma:injectivity} and the fact that $\md$ has negative value on vertices in $V$, that any global section of $L$ needs to vanish along $Y$. Thus we have an identification \[H^0(X,L) \simeq H^0 \left(Y^c, L|_{Y^c}\left(-\left(Y \cap Y^c\right)\right)\right).\] Now let $L'$ be a line bundle that restricts to $L|_{Y^c}(-(Y \cap Y^c))$ on $Y^c$ and to $L|_{Y}((Y \cap Y^c))$ on $Y$. Its multidegree by construction equals $\md' = \md - \underline t_{V_n}$. The restriction map \[H^0(X, L') \to H^0(Y, L'|_{Y^c}) \]
    has kernel \[H^0\left(Y, L'|_{Y}\left(-\left(Y \cap Y^c \right)\right)\right) = H^0(Y, L|_{Y}) = 0.\] 
    Hence \[h^0(X,L') \leq h^0(Y, L'|_{Y}) = h^0(Y, L|_{Y}((-(Y \cap Y^c)) = h^0(X,L), \]
    as claimed.

    Let $M:=L^{-1}\otimes L'$, which is a line bundle of multidegree $-\underline t_{V_n}$. Tensor product with $M$ induces a bijection $\varphi:\Pic^{\ud}(X)\to \Pic^{\ud'}(X)$. Arguing as above we get that $h^0(X,\varphi(N))\leq h^0(X, N)$ for all $N\in \Pic^{\ud}(X)$, so $\rmax(X, \md') \leq  \rmax(X, \md).$

    Repeating this construction eventually gives an effective multidegree $\ud''$ by \cite[Algorithm 3.10]{C23}, for which $\rmax(X, \md'') \leq  \rmax(X, \md)$, and the claim follows.
\end{proof}

\begin{rmk}
A consequence of Theorem \ref{thm:effective realization} is that, given an effective divisor class $\delta$, we can compute the uniform algebraic rank by computing $\rmaxx(G, \md)$ for finitely many representatives $\md$ of $\delta$.
\end{rmk}

\section{Clifford representatives}

In this section, we give an application of the properties of the uniform algebraic rank established in Section~\ref{sec:uniform algebraic rank} by showing that every divisor class in the special range $0 \leq d \leq 2g - 2$ contains Clifford representatives. In the remainder of the section we then discuss the question of identifying such representatives explicitly. 

\subsection{Existence of Clifford representatives}

In \cite{CLM} the inequality $\ralg(G, \delta) \leq r_G(\delta)$ and the fact that $r_G$ satisfies the Clifford inequality were used to show that on any nodal curve $X$ with dual graph $G$ there exists a divisor $\md$ on $G$ such that every line bundle on $X$ of multidegree $\md$ satisfies the Clifford inequality. Here we are interested in the following stronger notion: 

\begin{defn}
    Let $\delta$ be a divisor class on $G$. We call $\md \in \delta$ a \emph{Clifford representative} if every line bundle $L$ of multidegree $\md$ on any curve $X$ with dual graph $G$ satisfies the Clifford inequality, \[r(X, L) \leq \frac{d}{2}.\] 
\end{defn}

If $G$ is a graph with only one vertex $v$, then $X$ is an irreducible nodal curve and the fact that all line bundles $L$ of degree $d$ on $X$ satisfy the Clifford inequality is shown in the appendix in \cite{EKS} (actually in loc. cit. the authors show that the Clifford theorem holds more generally for torsion-free rank $1$ sheaves on integral curves). The existence of a Clifford representative is also known for weightless and trivalent graphs $G$ by \cite[Proposition 4.6]{CLM}.

In \cite[Question 4.7]{CLM} the authors ask whether such Clifford representatives always exist. 
Our main result in this section answers this question affirmatively:

\begin{thm} \label{thm:existencereps}
    Let $\delta$ be a divisor class of degree $0 \leq d \leq 2g -2$ on a graph $G$ of genus $g$. Then $\delta$ contains a Clifford representative.
\end{thm}

\begin{proof}
    Let $\md \in \delta$ be a multidegree that realizes the minimum in the definition of $r^{\mathrm{ALG}}(G, \delta)$. We claim that $\md$ is a Clifford representative. 
    
    Indeed, by Corollary~\ref{cor:CliffALG}, we have \[r^{\mathrm{ALG}}(G, \delta) \leq \frac{d}{2}.\] 
    Since $\md$ realizes the minimum in the definition of $r^{\mathrm{ALG}}(G, \delta)$, we have $r^{\mathrm{ALG}}(G, \delta) = r^{\mathrm{MAX}}(G, \md)$ and thus \[r^{\mathrm{MAX}}(G, \md) \leq \frac{d}{2}. \]
    Since we take the maximum in the definition of $r^{\mathrm{MAX}}(G, \md)$ while varying $X$ in $M_G$, this in turn implies 
    \[r^{\mathrm{max}}(X, \md) \leq \frac{d}{2}\]
    \emph{for all} $X \in M_G$. 
    
    Finally, since we take the maximum in the definition of $r^{\mathrm{max}}(X, \md)$ while varying $L$ in $\Pic^{\md}(X)$, this yields 
    \[r(X,L) \leq \frac{d}{2} \]
    \emph{for all} $X \in M_G$ \emph{and} $L \in \Pic^{\md}(X)$. Thus $\md$ is a Clifford representative by definition. 
\end{proof}

By the proof of Theorem~\ref{thm:existencereps}, any divisor $\md \in \delta$ that realizes $r^{\mathrm{ALG}}(G, \delta)$ is a Clifford representative. We saw in Theorem~\ref{thm:effective realization} that we may assume that such a divisor is effective. Furthermore, in some special cases representatives that realize $r^{\mathrm{ALG}}(G, \delta)$ are known, for example by Corollary~\ref{coroteob} (4), semibalanced divisors $\md$ realize the uniform algebraic rank if $d = 0$ or $d = 2g -2$. In the general case, identifying such representatives is wide open.

\subsection{Explicit Clifford representatives}
\label{sec:explicit clifford}

In this section, we give an explicit description of Clifford representatives for a large class of graphs. 
The construction of such Clifford representatives will distinguish between two cases, depending on whether a divisor class is special or not. We begin by introducing the necessary definitions. Recall that we set $\md^* = \underline k_G - \md$ and that we denote by $\delta^*$ the class of $\md^*$. 

\begin{defn}
    Let $G$ be a graph.
    \begin{enumerate}
        \item A divisor $\md \in \Div(G)$ is \emph{uniform} if both $\md$ and $\md^*$ are effective.
        \item A divisor class $\delta \in \Pic(G)$ is \emph{special} if both $\delta$ and $\delta^*$ are effective.
    \end{enumerate}
\end{defn}

Explicitly, a divisor $\md$ is uniform if for any vertex $v \in V(G)$, \[0 \leq \md_v \leq 2 \omega(v) - 2 + \val(v).\] 
We have $0 > 2 \omega(v) - 2 + \val(v)$ if and only if $\omega(v) = 0$ and $\val(v) = 1$ (recall that we assume $G$ to be connected). Thus there exist uniform multidegrees on a graph $G$ if and only if $G$ does not contain any vertex $v$ with $\omega(v) = 0$ and $\val(v) = 1$, i.e., $G$ is semistable.

\begin{rmk}
    As an aside, we observe that if $\delta$ is not special, then $\ralg(G, \delta) = \ralgg(G, \delta) = r_G(\delta)$, since $\ralg(G, \delta) \leq \ralgg(G, \delta) \leq r_G(\delta)$ and all three notions of rank satisfy the Riemann-Roch theorem.
\end{rmk}

A divisor $\md$ with negative degree $d$ cannot be effective. Since the degree of the residual $\md^*$ is $2g - 2 - d$, it follows that the degree of a uniform divisor satisfies $0 \leq d \leq 2g - 2$. The same holds for the degree of a special divisor class.

Clearly, the class $\delta$ of a uniform divisor $\md$ is special. It is not difficult to see that the converse need not be true, cf. \cite[Example 5.3]{C23}.

If the class $\delta$ is not special, it is easy to describe Clifford representatives: 

\begin{lma} \label{lma:non-effective}
    Let $\delta$ be a divisor class of degree $0 \leq d \leq 2g-2$ on a graph $G$ of genus $g$.
    \begin{enumerate}
        \item If $\delta$ is not effective, then the $v$-reduced divisor $\md$ in $\delta$ is a Clifford representative for any $v \in V(G)$. 
        \item If the residual $\delta^*$ of $\delta$ is not effective, then the representative $\md \in \delta$ whose residual divisor $\md^*$ is $v$-reduced is a Clifford representative for any $v \in V(G)$. 
    \end{enumerate}
    
\end{lma}

\begin{proof}
    Assume first that $\delta$ is not effective. Since the $v$-reduced representative $\md$ is by definition effective away from $v$, it follows that $\md$ is not effective on $v$. By Lemma~\ref{lma:injectivity} this means that any line bundle $L$ with multidegree $\md$ on a curve $X$ with dual graph $G$ does not admit a non-trivial global section, i.e., $h^0(X,L) = 0$. Such a line bundle satisfies the Clifford inequality since $d \geq 0$. 

    Assume next that the residual class $\delta^*$ is not effective. Let $\md \in \delta$ be such that its residual $\md^*$ is $v$-reduced. Then, arguing as above, any line bundle with multidegree $\md^*$ does not admit non-trivial global sections. Given a line bundle $L$ with multidegree $\md$ on a curve $X$ with dual graph $G$, the residual $K_X \otimes L^{-1}$ of $L$ has multidegree $\md^*$, where $K_X$ is the dualizing sheaf. Hence $h^0(X, K_X \otimes L^{-1}) = 0$. By the Riemann-Roch Theorem, we get
    \[h^0(X,L) = d - g + 1 + h^0(X, K_X \otimes L^{-1}) = d -g + 1 .\] 
    Finally, since $d \leq 2g - 2$ we have $\frac{d + 2}{2} \leq g$ and thus \[ h^0(X,L) = d -g + 1 \leq  d - \frac{d + 2}{2} + 1 < \frac{d}{2} + 1.\] 
\end{proof}

\medskip

The situation for special classes $\delta$ is more complicated. We will use \cite[Theorem 1.1]{Clifford} for this case, which gives a generalization of the Clifford inequality if the multidegree $\md$ of $L$ is uniform. In general, the bound of \cite[Theorem 1.1]{Clifford} is weaker than the classical Clifford inequality and equality in this weaker bound is achieved on any nodal curve $X$. Under certain assumptions on $G$, however, the two bounds coincide. To formulate the precise condition, we first introduce some notation.

Recall that a bridge of a graph $G$ is an edge whose removal disconnects the graph. Denote by $G^{\Br}$ the graph obtained from $G$ by contracting all edges that are not bridges, by construction a tree. We call $G$ a \emph{chain of $2$-edge connected components}, if $G^{\Br}$ is a chain; i.e., if it does not contain any vertex of valence greater than $2$. We note as a special case that if $G$ contains no bridges, it is a chain of $2$-edge connected components.

\tikzset{every picture/.style={line width=0.75pt}}
\begin{figure}[ht]	
\begin{tikzpicture}[x=0.7pt,y=0.7pt,yscale=-0.7,xscale=0.7]
\import{./}{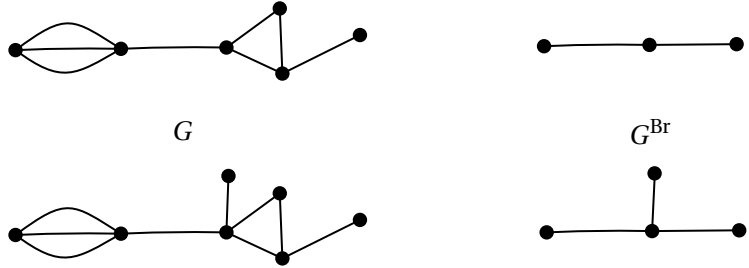}
\end{tikzpicture}		
\caption{On the left, two graphs $G$; on the right, their associated graph $G^{\Br}$ obtained by contracting all edges that are not bridges. The top one is a chain of $2$-edge connected components, the bottom one is not.}
\label{fig2}
\end{figure}

The chains of $2$-edge connected components are precisely the dual graphs for which the bound in \cite[Theorem 1.1]{Clifford} coincides with the classical Clifford inequality. Recall that we call a graph $G$ semistable if $G$ does not contain any vertex $v$ with $\omega(v) = 0$ and $\val(v) = 1$.

\begin{prop}\label{prop:uniform}
    A semistable graph $G$ is a chain of $2$-edge connected components if and only if every uniform divisor is a Clifford representative. 
\end{prop}

\begin{proof}
    This is part of \cite[Theorem 1.1]{Clifford} and \cite[Theorem 1.3]{Clifford}.
\end{proof}

Proposition~\ref{prop:uniform} allows us to describe Clifford representatives in special divisor classes, provided that $G$ satisfies the assumptions of the proposition and the class contains a uniform representative. In general, not every special divisor class contains a uniform representative. Our final ingredient is a result that gives a sufficient criterion for the existence of uniform representatives in special divisor classes \cite[Theorem B]{C23}. We again state a weaker version adapted to our purposes:

\begin{prop} \label{prop:existence uniform}
    Let $G$ be a graph such that every vertex $v \in V(G)$ with $\omega(v) = 0$ is adjacent to a loop edge. Then every special class $\delta \in \Pic^d(G)$ with $0 \leq d \leq 2g - 2$ contains a uniform representative.  
\end{prop}

Combining the previous results, we obtain:

\begin{thm}\label{thm:main_detailed}
    Let $G$ be a chain of $2$-edge connected components and assume for all vertices $v \in V(G)$ if $\omega(v) = 0$ then $v$ is adjacent to a loop. Let $\delta \in \Pic^d(G)$ be a divisor class of degree $0 \leq d \leq 2g - 2$ on $G$. Then $\delta$ contains a Clifford representative given by:
\begin{enumerate}
    \item a uniform divisor, if $\delta$ is special; 
    \item a $v$-reduced divisor for some vertex $v \in V(G)$, if $\delta$ is not special and not effective;
    \item a divisor $\md$ whose residual $\md^*$ is $v$-reduced for some vertex $v \in V(G)$ if $\delta$ is effective and not special.  
\end{enumerate}
\end{thm}

\begin{proof}
    Suppose first that $\delta$ is special. Since we assume that for all vertices $v \in V(G)$, if $\omega(v) = 0$ then $v$ is adjacent to a loop, $G$ is in particular semistable. Furthermore, by Proposition~\ref{prop:existence uniform}, $\delta$ needs to contain a uniform representative $\md$. 
    By assumption, $G$ is in addition a chain of $2$-edge connected components, and hence any uniform divisor $\md$ is a Clifford representative by Proposition~\ref{prop:uniform}.

    The cases in which $\delta$ is not special are covered by Lemma~\ref{lma:non-effective}.
\end{proof}
  
\providecommand{\bysame}{\leavevmode\hbox to3em{\hrulefill}\thinspace}
\providecommand{\MR}{\relax\ifhmode\unskip\space\fi MR }
\providecommand{\MRhref}[2]{%
  \href{http://www.ams.org/mathscinet-getitem?mr=#1}{#2}
}
\providecommand{\href}[2]{#2}

\end{document}